\documentclass[a4paper,10pt,leqno, english]{amsart}
\title[Nonlinearity]
{Nonlinearity, Proper Actions and  Equivariant Stable  Cohomotopy.}

\author{No\'{e} B\'{a}rcenas }
        \address{Hausdorff  Center for  Mathematics\\
        Universit\"at  Bonn\\ 
        Endenicherstra\ss e 60 53115  Bonn, Germany}
         \email{barcenas@math.uni-bonn.de}
        \urladdr{http://wwwmath.uni-bonn.de/people/barcenas}

         \date{today}

\usepackage{hyperref}
\usepackage{pdfsync}
\usepackage{calc}
\usepackage{enumerate,amssymb}
\usepackage[arrow,curve,matrix,tips,2cell]{xy}
  \SelectTips{eu}{10} \UseTips
  \UseAllTwocells

\DeclareMathAlphabet\EuR{U}{eur}{m}{n}
\SetMathAlphabet\EuR{bold}{U}{eur}{b}{n}

\makeindex             

\theoremstyle{plain}
\newtheorem{theorem}{Theorem}[section]

\newtheorem{proposition}[theorem]{Proposition}

\theoremstyle{definition}
\newtheorem{definition}[theorem]{Definition}
\newtheorem{example}[theorem]{Example}
\newtheorem{condition}[theorem]{Condition}
\newtheorem{remark}[theorem]{Remark}

{\catcode`@=11\global\let\c@equation=\c@theorem}

\newcommand{\comsquare}[8]                   
{\begin{CD}
#1 @>#2>> #3\\
@V{#4}VV @V{#5}VV\\
#6 @>#7>> #8
\end{CD}
}

\newcommand{\xycomsquare}[8]                   
{\xymatrix
{#1 \ar[r]^{#2} \ar[d]^{#4} &
#3 \ar[d]^{#5}  \\
#6\ar[r]^{#7} &
#8
}
}

\newcommand{\ind}{\operatorname{ind}}

\newcommand{\E}{\rm E}

\newcommand{\EfinG}{\underline{\E} G}

\newcommand{\Komp}{\mathcal{COM}}

\newcommand{\ideal}{\hat{{\bf I}}}

\newcommand{\pt}{\{\bullet\}}

\newcommand{\B}{\rm B}
\newcommand{\higherlim}[3]{{\setbox1=\hbox{\rm lim}
        \setbox2=\hbox to \wd1{\leftarrowfill} \ht2=0pt \dp2=-1pt
        \mathop{\vtop{\baselineskip=5pt\box1\box2}}
        _{#1}}^{#2}#3}

\begin{document}

\typeout{----------------------------  linluesau.tex  ----------------------------}

\maketitle

\typeout{------------------------------------ Abstract ----------------------------------------}

In this  article  we  extend  the  classical  definitions  of  equivariant cohomotopy theory  to  the  setting  of  proper  actions  of  Lie  groups. We  combine  methods  originally  developed  in the  analysis  of nonlinear  differential  equations, mainly  in connection  with  Leray-Schauder theory,  and  on the  other  hand  from   developments  of  equivariant $K$-Theory  by  N.C. Phillips. We  prove  the  correspondence  with a  previous construction  of  W. L\"uck by constructing an  index.

 As  an  illustration of  these  methods, we introduce  a Burnside ring  defined  in analytical terms. With this  definition, we extend  a  weak  version  of  the Segal Conjecture to a certain family  of  Lie groups and  comment  the  relation to an  invariant in Gauge Theory, due  to Bauer  and  Furuta.
Keywords: Equivariant Cohomotopy, Proper Actions, Nonlinear  Analysis, Segal Conjecture. 
Subject Classification: 55P91. Secondary: 2205,47414.

\setcounter{section}{-1}

\section{Introduction.}
 \typeout{-------------------------------   Section 0: Introduction --------------------------------}

We extend  the  definition of  equivariant  cohomotopy  to  the  setting  of proper  actions  of  Lie  groups.  We  define  equivariant  cohomotopy  in  terms of  certain  nonlinear  perturbations  of  Fredholm  morphisms  of  Hilbert  bundles over  a  given proper  $G$-CW  complex. 

This  is  motivated by  the  fact that  the  technical  difficulties  involved  in the  proof  of  excision  for  equivariant  cohomology theories, where  the  equivariance  group  is  neither discrete, nor  compact Lie  cannot  be solved  with  constructions using  finite  dimensional  $G$-vector  bundles or  finite  dimensional  representations \cite{phillipsktheory}, \cite{emersonmeyer}. 

We  prove (Theorem \ref{theorem schwartzindex}) that  our  invariants  generalize  previous definitions,  such  as  that  of  W. L\"uck  \cite{lueckeqstable} in the  context  of  proper  actions  of  discrete  groups on  finite  $G$-CW  complexes. The  proof  has its  roots in methods  employed   in nonlinear  analysis \cite{schwartz}, \cite{doldfixpunkt}, \cite{marzantowicz}. But  is  motivated  and  impregnated   by   ideas,  methods  of  proof  and  constructions   concerning gauge  theoretical  invariants  in   \cite{bauer}.  

We  illustrate  the applications  of  our  methods by  several  examples. We  comment  the  potential  utility  of  this  approach  by  generalizing a  gauge theoretical  invariant  of $4$-dimensional smooth  manifolds, due  to  Bauer and Furuta  to  allow  proper  actions  of  Lie  groups  on  four-manifolds.  Finally, we  introduce a  Burnside ring  in  operator  theoretical  terms  and  as  a test  for  the  suitability of this definition. After  some  computational remarks we  verify  the  extension  of  a  weak  form  of  the  Segal Conjecture  for  a  certain class of   non-compact  Lie  groups.

 This  paper  is  organized  as  follows: in  section \ref{sectionreview}, a  review  of  equivariant  cohomotopy  for  compact  Lie  groups  is  given. Preliminaries  on  Hilbert Bundles  and  Fredholm  morphisms  between  them  are  given  in section \ref{sectionhilbert}. Section \ref{sectioncocycles} defines  the  analytical  objects  which  define  equivariant  cohomotopy  groups  for  non compact  group  of  symmetries, and  examples  are  considered  in section \ref{sectionexamples}. 
The  properties  of  equivariant  cohomotopy  are  discussed  in  \ref{sectioncohomology},  the  relation  to  previous  work  is  stablished  in  section  \ref{sectionindex}  via  a  construction  of  a  natural  isomorphism,  the parametrized  Schwarz  index. In the  last  section \ref{sectionburnside}, a Burnside  ring  for  noncompact  Lie  groups  is  introduced,  and  extensions  of  the  Segal  conjecture  are  discussed  in this  new  setting.

\subsection{Aknowledgments}
 The  author  would  like  to thank the  Mexican  Council  for  Science  and  Technology, CONACYT  for  economical  support in  terms  of  a  Ph.D grant. The  main  results  of this  note developed  from the  corresponding dissertation, defended  at  the  Westf\"alische  Wilhelms-Universit\"at M\"unster.  This  work   was also  supported  by  the  SFB 878 at  the  University  of  M\"unster, Wolfgang's  L\"uck Leibnizpreis, and the Hausdorff  Center  for  Mathematics  in Bonn.

\section{Review  of Equivariant  Cohomotopy  for  compact  Lie Groups. }\label{sectionreview}
\typeout{-------------------------------   Section 1: Review  of Equivariant  Cohomotopy  for  compact  Lie Groups.  --------------------------------}
In the case  of  compact  Lie  groups equivariant  cohomotopy  is  a  $RO(G)$- graded  equivariant  cohomology  theory, in the  sense  of \cite{mayeqhomcohom}. 
\begin{definition}
Let  $G$ be  a  compact Lie group  and  $X$ a  $G$-CW complex. For  any  representation $W$,  form  the  one-point  compactification $S^{W}$ and  define  the  set  of  equivariant  and  pointed  maps  $\Omega^{W}S^{W}={\rm Map}_{G}(S^{W}, S^{W}) $.  The  equivariant  cohomotopy  group  in degree $V=V_{1}-V_{2}$, where  $V_{i}$  are  finite  dimensional  real representations,   is  defined  to  be  the  abelian  group constructed  as  the  colimit  of the system  of  homotopy  classes  of  maps

 $$\pi_{G}^{V}(X)={\rm  colim}_{W} [S^{V_{1}}\wedge X_{+},\Omega^{W}S^{W\oplus  V_{2}}]$$
 where  the  system  runs  along  a  complete  $G$-universe,  that is  a  Hilbert  space  containing  as  subspaces all  irreducible representations,  where the  trivial  representation  appears  infinitely  often.    
\end{definition}

 Finite  dimensional representations  are involved  in this  definition  in  a  crucial  way  and  this is the  main  handicap  to  extend  equivariant  cohomotopy  to   more  general  settings. In fact, pathological  examples  from  group  theory \cite{olshanskii} provide  finitely  generated, discrete   groups  for  which  all finite  dimensional  representations  over  $\mathbb{R}$, or  $\mathbb{C}$ are trivial. 
Precisely:

\begin{example}
 Let  $G$  be  a  finitely  presented  group. $G$  is  said  to  be  residually  finite  if  for  every  element  $g\neq 1$, there  exists  a  homomorphism $\varphi$  to a  finite  group  mapping  $g$  to  an   element  different  from $1$.  
The  maximal  residually  finite  quotient  $G_{mrf}$  of  $G$  is  the  quotient  by  the  normal  subgroup consisting  of   the  intersection of  all  subgroups  of  finite  index. In  symbols: 
 $$G_{mrf}=\frac{G}{\cap_{(G:H)<\infty}H}$$ 
This  is  a  residually  finite  group,  characterized  by  the  property  that every  group  homomorphism  to a  residually  finite  group  factorizes  trough the  quotient map. Recall \cite{malcev}, that if  $G$  is  a  finitely  generated  subgroup  of  $Gl_{n}(F)$  for  some  field, then  $G$  is  residually  finite. This   means  that  in this  situation,  every finite  dimensional  representation of  $G$  is  induced   from  one  of $G_{mrf}$. 
An  example  of  Olshanskii, \cite{olshanskii} gives  for  every  prime  $p>10^{75}$ a  finitely  generated, infinite  group all  of  whose  proper  subgroups  are  finite  of order  $p$. That  means  that  $G$  does  not  contain  proper  subgroups  of  finite  index,  hence  $G_{mrf}=\{e\}$  and  consequently  every finite  dimensional  representation  of  $G$ is  trivial.     
\end{example}


\section{Preliminaries on  Hilbert Bundles  and  Fredholm  Morphisms}  \label{sectionhilbert}
\subsection{Elementary Equivariant  Topology}

We  recall  first  some basic definitions  and technical  facts of  equivariant  Topology. 
 
\begin{definition}
Let $G$  be  a second  countable, Hausdorff locally  compact  group. Let  $X$  be  a  second  countable, locally  compact Hausdorff space. Recall  that  a  $G$-space  is  proper  if  the  map 
$$\underset{\overset{\theta_{X}} {(g,x)\mapsto (x,gx)} } { G\times  X\to X\times X}$$ 
is  proper. 
\end{definition}

\begin{remark}
In the  case  of  Lie groups,  a  proper  action  amounts  to  the  fact  that  all  isotropy  subgroups  are  compact  and  that   a  local  triviality  condition, coded  in the  Slice  Theorem is  satisfied \cite{palais}. Specializing  to  Lie  groups  acting  properly on  $G$-CW  complexes, (see  the  definition below), these  conditions boil  down  to the  fact  that  all  stabilizers  are  compact \cite{luecktransformation},  Theorem  1.23.  In particular  for  an action  of  a  discrete  group  $G$  on  a  $G$-CW  complex,  a  proper  action  reduces   to  the  finiteness of  all  stabilizer  groups.
\end{remark}

\begin{definition}
Recall  that  a  $G$-CW  complex structure  on  the  pair $(X,A)$  consists  of a  filtration of  the $G$-space $X=\cup_{-1\leq n } X_{n}$ beginning  with $A$ and  for   which  every  space   is inductively  obtained  from  the  previous  one   by  attaching  cells  in pushout  diagrams  
$$\xymatrix{\coprod_{i} S^{n-1}\times G/H_{i} \ar[r] \ar[d] & X_{n-1} \ar[d] \\ \coprod_{i}D^{n}\times G/H_{i} \ar[r]& X_{n}}$$   
We  say  that a  proper  $G$-CW complex  is  finite  if  it  consists  of  a  finite  number  of  cells  $G/H\times  D^{n}$. 
\end{definition}

The  following result  enumerates  some  facts  which   will be  needed in the  following, which  are  proven  in chapter one  of \cite{luecktransformation}:
\begin{proposition}
Let  $(X,A)$ be  a  proper  $G$-CW  pair 
\begin{enumerate}
 \item{The  inclusion  $A\to X$  is  a   closed cofibration.}
 \item{$A$  is  a neighborhood  $G$-deformation  retract, in the  sense  that  there  exists  a neighborhood $A\subset U$, of  which $A$ is a  $G$-equivariant  deformation retract. The  neighborhood  can be  chosen to be  closed  or  open.   }
\end{enumerate}
\end{proposition}

We  recall the  notion  of  the  classifiying  space  for  proper  actions: 
\begin{definition}
 A  model  for  the classifying  space  for  proper  actions is  a $G$-CW  complex  $\EfinG$  with  the  following  properties: 
\begin{itemize}
 \item{All  isotropy  groups  are  compact}
\item{For  any  proper  $G$-CW  complex $X$  there  exists  up  to $G$-homotopy  a  unique $G$-map $X\to \EfinG$  }
\end{itemize}
\end{definition}

The  classifiying  space  for  proper  actions  always  exists, is  unique  up  to $G$-homotopy  and  admits  several  models. 
The  following  list   contains  some  examples. We  remite  to \cite{lueckclassifying} for  further  discussion.

 \begin{itemize}
\item{If  $G$  is  a  compact  group,  then  the  singleton  space  is  a  model  for  $\EfinG$. }
\item{Let $G$ Be  a  group  acting  on  a     ${\rm Cat}(0)$ space  $X$. Then  $X$ is a model for  $\underline{{\rm
  E}G}$. }
\item{Let  $G$  be  a  Coxeter  group. The  Davis complex  is  a model for   $\underline{ {\rm E} G }$.}
\item{Let  $G$ be  a   mapping  class  group  of  a  surface   Teichm\"uller  space  is a  model  for $\EfinG$. } 
\end{itemize}

\subsection{Fredholm  Morphisms}

We  begin  by describing cocycles  for equivariant  cohomotopy. They 
 are  defined  in terms  of  certain  nonlinear  operators  on real 
 $G$-Hilbert  bundles, so  we  briefly  recall  some  well-known
 facts on then  and  their  morphisms.
 A  comprehensive  treatment  of  Hilbert   bundles and   their  linear  morphisms  them  is 
 given  in the  book \cite{phillipsktheory}. For  matters related  to  real $C^{*}$-algebras  we  refer to  the  text \cite{schroeder}, in particular  for connections  with  Kasparov $KK$-theory.

\begin{definition}

Let  $X$ be  a locally  compact, Hausdorff   proper $G$-space. A
Banach  bundle  over $X $ is a locally  trivial  fiber  bundle  $E$
with  fiber modeled  on  a Banach  space $H$, whose  structure  group is the
set of  all  isometric  linear bijections   of $H$ with  the  strong topology.

If  $H$  is  a  (real) Hilbert  space, we   will speak  of  a  Hilbert
bundle. A  $G$-Hilbert  bundle is a  Hilbert bundle $p:E\to X$ endowed  with  a continuous action  of the  locally  compact  group $G$
in the  total  space.  The  map $p$ -called  the  projection- is  assumed  to be  $G$-equivariant
and  the  action on the  total  space  is  given  by  linear isometric  bijections. 
\end{definition}
\begin{definition}[Linear  Morphism]
Let  $E$ and $F$ be  Hilbert bundles over $X$. A linear morphism from
$E$ to $F$  is an equivariant,  continuous function $t:E\to F$ covering the
identity  on $X$ consisting  of  bounded  operators on  fibers and  for which the  fiberwise  adjoint  $t^{*}$
defined  by $\langle t^{*}x,y\rangle= \langle x,ty\rangle$  exists and  is  continuous. The  support of  a
morphism $t$ is  the  set $\{x \in X\mid tx\neq 0\}$.

\end{definition}

\begin{definition}[Compact linear  morphism]
A linear  morphism $t: E\to F$  is a compact morphism if  it is
fiberwise  compact  in the  usual  sense (that  is, for  every $x$,
$t_{x}.E_{x}\to F_{x}$ maps bounded  sets  to  relatively  compact
sets) and   in adition, for  every  point $x\in X$, there  exist  local
trivializations $a:E\mid_{U}\to U\times  E_{x}$, $b: F\mid_{U}\to
U\times  F_{x}$ such  that $bta^{-1}: U\times E_{x}\to U\times F_{x}$
is given  by  an  expresion  $(y,x)\mapsto (y, \psi(x))$ for  a  norm
continuous  map $\psi: U\to L(E_{x}, F_{x})$ into  the  bounded
linear  maps between the  fiber  over $x$. 
\end{definition}

\begin{definition}[Fredholm morphism]
A linear  morphism $t: E \to F$  is  said  to  be  Fredholm if there
exists a  morphism $s:F\to  E$ such that $st-1$ and  $ts-1$ are
compact  morphisms with compact supports. A  Fredholm  morphism  is  said  to  be  essentially  unitary  if   one  can  take $t=s^{*}$ in  the  definition.  We  recall the  existence  of $G$-invariant riemannian metrics on  vector  bundles  over  proper $G$-spaces,  which  is  proved  for  instance  in \cite{palais} as   consecuence  of  the  slice  theorem. 
\end{definition}

\begin{remark}
\begin{enumerate}

\item{Notice  that  we choose  the structure  group  of  our  bundles
  to  be  the isometric  bijections with  the  strong  operator
  topology. Other  choices, like the  weak $*$-topology  would not  give
  \emph{enough morphisms}, as  pointed  out by  Phillips in \cite{phillipsktheory}, chapter 9.}

\item{ Notice  that  we  forced  the  existence  of  adjoints  for  our
  morphisms. This  is  technically  convenient. We  also  assume  that adjoint operators exist and   are  always  continuous  in the  operator  norm, and  the  same  for morphisms  of  Hilbert  bundles.}

\end{enumerate}
\end{remark}

We  now  ennumerate  a  collection  of  facts  on linear  morphisms  and  $G$-Hilbert  Bundles which  we  will  use  later. 

\begin{proposition}[Proper  stabilization theorem] \label{proposition proper stabilization}
Let $E$ be  a  $G$-Hilbert  bundle over  a  proper  $G$-space. 
Denote  by $\mathcal{H}$  the  numerable Hilbert  space   consisting  of  the  numerable  sum of  the  space  of  square  integrable  functions  in  $G$, in symbols $\mathcal{H}=\oplus_{n=1}^{\infty}L^{2}(G)$. Let $\mathcal{H}\times X$  be  the associated  trivial Hilbert $G$-bundle. There exists  an  equivariant linear  isomorphism of $G$-Hilbert  Bundles
$$E\oplus \mathcal{H}\times  X\cong \mathcal{H}\times X$$
\end{proposition}
 \begin{proof}
 Theorem 2.9, p. 29 of \cite{phillipsktheory}, Theorem 2.1.4  , p. 58  of  \cite{schroeder} in the real case. We  point  out  that  Phillips  realizes   this  isomorphism to  be  an  adjoinable morphism   between the  $G$-$C_{0}(X)$- Hilbert modules $\Gamma(E)\otimes_{  C_{0}(X)} \Gamma(\mathcal{H}\times  X)$ and  $\Gamma(X\times \mathcal{H})\otimes_{C_{0}(X)}  C_{0}(X)$. The  identification  of  such  a  morphism  with an  isomorphism  of  $G$-Hilbert  bundles  is  consequence  of  lemma  1.9  in \cite{phillipsktheory}.    
 \end{proof}

Next,  we  modify  Phillip's definition  of complex equivariant  $K$-theory for  proper
actions  of  locally compact  groups, \cite{phillipsktheory} to  allow  real  cocycles. The  main reference  for  technical  isssues  concernig  the passage  to  real $K$-theory  is  \cite{schroeder}.  

\begin{definition}
The  real  equivariant  $K$-theory  of  the  proper  and finite   $G$-CW  complex $X$  
$KO_{0}^{G}(X)$  is  represented  by cocycles  $(E,F, l)$, where $E$
  and   $F$  are  real  $G$-Hilbert  bundles and  $l:E \to F$  is  a
  fiberwise \emph{linear} real  Fredholm  morphism.   A  cocycle is  said
  to be \emph{trivial} if   $l$ is fibrewise  unitary. Two cocycles
  $(E_{i}, F_{i}, l_{i})_{i=0,1}$ are   equivalent  if there  exists
  a  trivial cocycle  $\tau$ such  that $l_{0}\oplus  \tau=
  l_{1}\oplus \tau $  is   homotopic  to  a trivial  morphism. 
\end{definition} 

We  now  enumerate  two consequences  of the  proper  stabilization theorem, which are fundamental  for  veryfing  excision  in  Phillips'  construction of  equivariant  $K$-theory.

\begin{proposition} \label{proposition linear extension}
 Let  $i:U\to X$ be   the  inclusion of a  $G$-invariant, open subset of $X$. Let $(E, F, t)$  be  a  linear  cocycle  over  $U$ such  that  $t$  is fibrewise bounded and  has a  bounded  Fredholm inverse. Then, there  exists a linear  cocycle $(X\times \mathcal{H}, X\times\mathcal{H}, r)$  such that the  classes  $i^{*}(r)$ and $t$ agree after adding  a  unitary  linear  cocycle. 
\end{proposition}

\begin{proof}
Proposition  5.9, p. 74  in  \cite{phillipsktheory}. We  recall  that  the constructed  classes  agree  after  application  of  the  proper  stabilization  theorem, for  the real  modification  see  \cite{schroeder}, theorem 2.1.4 p.58 .

\end{proof}

\begin{proposition} \label{proposition Hilberthomotopy}
Let $X$ be a  proper $G$-CW  complex and $\varphi: E\to  E$ be  a  fibrewise  Fredholm  operator  defined  on  the  space  $X\times  I$. Denote by  $\varphi_{0}: E_{0} \to  E_{0}$  the   restriction  to  $X\times \{0\}$. Then: 
\begin{enumerate}

\item{There  exists  a  unitary cocycle  $\rho$ between $(E,E, \varphi)$  and  $(E_{0}, E_{0}, \varphi_{0})$. Moreover,  the  isomorphism  can be  taken  to   be  unitary  over a  fixed, invariant   subspace $A\subset X$. }
\item{Let  $A\subset X$  be  a  $G$-subcomplex and  $l:F\mid_{A}\to E \mid_{A}$ be  a   bounded  morphism. Then, there  exists a  linear cocycle   $(l^{'}, E, F)$, defined  over $X$ such  that  $i^{*}(l)$  and $l$  are  equivalent . }
\item{Let  $A\subset X$  be  a  $G$-invariant closed  subset   and  $U\subset X $ an  open  neighborhood  of  which $A$  is  a  deformation retract. Suppose  that  $(U\times K, U\times  K, l)$ is  an  essentially  unitary  cocycle  over  $U$, where  $K$ is a  strong continuous  unitary $G$-representation  in a  Hilbert  space. Then, there  exists  an  essentially unitary  cocycle $(X\times  H, X\times  H, F)$ such  that  $i^{*}(F)=l$.     }

\end{enumerate}

\end{proposition}
\begin{proof}
\begin{enumerate}
\item{As the involved  bundles  are  locally  trivial $E$ carries  the  weak  topology  with  respect  to  the set $p^{-1}(X_{i})$, where $X_{i}$  is  an element  of  the  $CW$- filtration  in  the  basis (cfr. lemma  1.26 in \cite{luecktransformation}). Hence, the  statement  reduces  to the  case  where  $(X, A)= (G/H\times  D^{n}, G/H\times S^{n-1} )$ and  $E$  has  the  form  $ G\underset{H}{\times} \mathcal{H}$,  for  a  given   strong- norm  continuous and unitary  representation of  the  compact  lie  group $H$   in a  separable real  Hilbert  space $\mathcal{H}$. Let  $\mathcal{U}_{c}(\mathcal{H})$  be  the  subspace  of  the  $H$-equivariant, unitary operators $u$  in $\mathcal{H}$, for  which  the  conjugation  with  an  arbitrary  element $h^{-1}uh$  is a  continuous   operator (recall  that  this  is  a  contractible  space after results of  Segal\cite{atiyahsegal}, Apendix 3 for  the compact  Lie case). Giving  an isomorphism $\rho$ as  described  above   amounts  to  give  a  map $\rho \in {\rm Map}(D^{n}\times  I, \mathcal{U}_{c}(\mathcal{H})) $ which  is  the  identity  on $G/H\times S^{n-1}\times  \{0\}$. There  is  no  obstruction  for  doing  this  because  the  inclusion $G/H\times  S^{n-1}\to G/H\times  D^{n}$ is  a  cofibration and  $\mathcal{U}_{c}(\mathcal{H})$ is  contractible. }
\item{It  follows  from the  reduction  to  a  cell  as  above  and  the  fact  that  for  a  compact  group $H$, the  space  of  $H$-equivariant  compact  operators  is  contractible.}
\item{It  follows  from the  restriction  to a  cell as  above  and  the  contractibility  of the  equivariant  unitary  group  of  a  Hilbert  representation for  a  compact group.}
\end{enumerate}
\end{proof}

\section{Cocycles  for  equivariant stable  cohomotopy}\label{sectioncocycles}

The  second  ingredient  for  our  construction of  cocycles for
stable  cohomotopy  are some  basic  notions  of  nonlinear
functional  analysis. References to  this topic
are  the books \cite{berger} and  \cite{deimling}. 
Applications to differential equations  can be  found  in  \cite{ize}, \cite{marzantowicz}.  From  this  section  on, all  groups are  assumed  to  be Lie,   actions    are  assumed  to  be  proper  and  all  spaces   are  finite  $G$-CW  complexes.

\begin{definition}[Cocycles  for  equivariant  Cohomotopy]
 Supose  that  $k:H\to  H$  is  a  (possibly  nonlinear)  compact  map in a  (real) Hilbert  space, in the  sense  that  $k$  sends  bounded  sets  into  relatively compact  sets.
  A  map  $f=t +k:H\to  H$  is  a  compact  perturbation  of the   real Fredholm  linear  morphism $t$  if  the   preimages  of  bounded sets under  $f$ are  bounded. Let  $t:E\to  F$ be  a   Fredholm  morphism  between  the $G$-Hilbert  bundles  over  the  proper $G$- space  $X$.  A compact  perturbation  of  $t$  is  a continuous, equivariant  map $f: E\to  F$, which  has  fibrewise the form  $t+k$  for  a  compact  perturbation  of  $t$. 
\end{definition}

\begin{definition}
Let  $X$ be  a locally  compact, proper  $G$-CW  complex and  let  $G$ be  a
locally  compact  group. Let $l$  be  a representative of   a real   linear  cocycle  representing  a   fixed class in $KO_{G}^{0}(X)$ in the  sense  of  Phillips. 

A  cocycle  for  the  equivariant  cohomotopy theory of  $X$, $\Pi_{G}^{[l]}(X)$  is a four-tuple  $(E, F, l, c )$ where 
\begin{itemize}
\item{$E$, $F$ are real  $G$- Hilbert bundle  over $X$, with a  linear, real  Fredholm
  morphism $l:E\to F$ in the  sense  of  Phillips.}

\item{A  map   $c:E\to F $, lifting  the identity and  possibly nonlinear  on  each  fibre, for  which  There exist  local
trivializations $a:E\mid_{U}\to U\times  E_{x}$, $b: F\mid_{U}\to
U\times  F_{x}$ such  that $bta^{-1}: U\times E_{x}\to U\times F_{x}$
is given  by  a continuous  expresion  $(y,x)\mapsto (y, k(x))$  consisting  of  possibly  nonlinear  compact  maps, in  a  way that  
 $l+c$  extends  to  a  map  between  one-point  compactification bundles.}
   
\end{itemize}

Two cocycles $(E, F, l, c) $ and  $(E^{'}, F^{'}, l^{'},   c^{'})$  are  equivalent  if
there  is  a linear, unitary  cocycle $(H,H^{'}, u)$ such that
$E\oplus H, F\oplus H,  l\oplus u$ and  $E^{'}\oplus  H^{'},
F^{'}\oplus H^{'} l^{'}\oplus  u$  are  unitary  equivalent  as  linear
cocycles, by  an  isomorphism  which intertwines  the compact
perturbations.

   Two cocycles $(E, F, l,  c)$ and $(E, F, l, c^{'})$  are
homotopic  if  there  exists  a homotopy  $H:S^{E}\times  I  \to S^{F}$, covering  the  projection  $X\times I \to X$,  pointed  over  every  fiber and  relative  to  $l$. That means,   $H$ has  the  form   $l+h$ for  a  certain  compact map  $h: S^{E}\times  I\to  S^{F}$.  . 
such  that $H\mid_{0}=l+c$, $H\mid_{1}= l +c^{'}.$

The  set  $\Pi^{[l]}_{G}(X)$ is  called  the $G$-equivariant  cohomotopy  of $X$ in degree $[l]$. 

\end{definition}

In  several  applications  in  analysis, \cite{marzantowicz}, \cite{furuta} cocycles  have  a  diferent, but  equivalent  shape:  

\begin{remark}[The ``no zeros  on the  boundary''-picture]
Equivariant  cohomotopy  also  can  be   described  as  four-tuples $(E,F,t,k)$
 where  $t: E\to F$ is   a real  Fredholm morphism between $G$-Hilbert  bundles, with  a  choice  of  an   invariant  riemannian metric on $E$, 
$k$  is  a  \emph{non   necessary linear}, compact   fibrewise
map defined  on the  unitary  disk $D(E)=\{x\in E\mid \vert x\vert\leq 1\}$ which  is   \emph {strongly  non zero}, in the  sense  that  $kx=0$  has  no  solutions  on the  boundary of  the unit disk $\partial D=\{x\in E\mid  \vert x\vert=1\}$.
Cocycles  are  said  to  be  equivalent  if  they  agree  after  adition  of  a   linear, unitary  Fredholm  morphism. 
The  homotopy  condition  is   defined  with  homotopies  which  not  only   are assumed  to  be  compact,  but  also  strongly  non-zero and relative  to $l$, that is,  a  homotopy  is  a  map  of  the  form $l+h$, where $h$  is a compact map  defined  on $DE\times I$,  for  which  for  any  $t$, the  corresponding  map  has  no zeros  on the  boundary  
\end{remark}

The  coincidence  of  both  approaches in the  case  $X=\pt$  with  an  action  of  a  compact  lie  group, as  for  example  in  \cite{marzantowicz}   follows  from the  fact  that a   strongly  non-zero  compact  perturbation amounts  to   maps  of pairs $(D(H),\partial D(H))\to (H,H-\{0\})$,  which  are  equivalent  via  excision  and  homotopy  equivalence  to  $(S^{H},S^{H})$,  via   the intermediary  pairs $(S^{H}, H-\{0\})$ and $(S^{H}, S^{H}-{\rm Int} D(H))$. Notice  that  the  point  at infinity in the  pointed  picture is  identified  with  the  point  $(0,1)$ in the ``no-zeros on the  boundary''-picture, if  we  identify  the  pair  $D(H)\subset H\oplus \mathbb{R}$ in the  usual  way, and  everything  holds  fibrewise  and  equivariant  over  a  proper  $G$-space.

The  main  difference  between  our  definition  an the  classical  one  is  the  fact  that  our  theory  is  graded   by  the  group  $KO_{0}^{G}(X)$, instead  of  the  representation  ring.  $KO^{0}_{G}(X)$-theoretical  graded  invariants  play  a  role  in \cite{telemanokonek},  \cite{bauer} and  \cite{carlossegal}.

The  following   remark   addresses  what  may  be a generalization of  the  classical $RO(G)$-grading:

\begin{remark}[$KO_{G}^{0}(\EfinG)$-graded cohomology theories] 
An  alternative  approach  to   equivariant cohomotopy (restricting  the  grading to a  generalization  of  the  representation  ring ) can  be  constructed  as  follows: 

Let $X$  be  a  proper $G$-CW complex  and let $r: X\to \EfinG$  be  the   unique  map  up  to  homotopy  to  the  classifying  space of  proper  actions.  
Cocycles  are   four-tuples  $(r^{*}(E), r^{*}(E), r^{*}(l), c)$  which  are  declared  to be  equivalent  if  they  agree after  performing  following operations: 

\begin{condition}\label{condition triviality} 
\begin{itemize}
\item{Unitary  embedding. For  any linear cocycle  $(E,E,L)$ and  any commutative  diagram 
$$\xymatrix{l:r^{*} (E_{0}) \to  r^{*}(E _{0})  \ar[r]^-{U} \ar[d] & L:E\to E \ar[d] \\ X \ar[r] & \EfinG}$$

where  the  map $U$  is  a unitary  embedding  over  every  point,  the  cocycles  $(r^{*}(E_{0}), r^{*}(E_{0}),r^{*}(l), c)$ and $(r^{*}(E), r^{*}(E), r^{*}(L), u \circ c)$ are  equivalent. }
\item{Pullback  of  homotopies  of  isometric  embeddings over  $\EfinG$.}
\item{ Proper  homotopy   of   nonlinear  compact  maps   relative   to  the  class of  $r^{*}(l)$.}   
\end{itemize}

\end{condition}

\end{remark}

\section{Examples  of  Cocycles for  equivariant  cohomotopy}\label{sectionexamples}
I this  section  we  will  give  some  of  the  most  common  examples  of  cocycles  in  equivariant  cohomotopy from the  analytical viewpoint. The  first four  of  them  are  classical and show an increasing  degree  of  generality. The  last  example   generalizes  an invariant due  to  Stefan  Bauer  and  Mikio  Furuta in Gauge  Theory and  is  not  indispensable  for  the  rest  of the  paper.

\begin{example}[Collapse  map  related to  the Thom-Pontrjagyn Costruction]
Let  $M$   be  a  framed, orientable  $k$-dimensional  manifold  with  an  embedding  $M\to U \subset \mathbb{R}^{n+k}$, where  $U$  is  a  tubular  neighborhood. The Thom- Pontrjagyin  construction   gives  a  proper map  $c:S^{\mathbb{R}^n+k}\to S^{\mathbb{R}^{n}}$ by  collapsing  the  complement  of  $U$  to  infinity  and  using  the  framing  to  project  onto $S^{\mathbb{R}^n}$. The  cocycle  $(\mathbb{R}^{n+k}, \mathbb{R}^{n}, \rm{proj}_{\mathbb{R}^{n}}, c)$ represents  an  element  in $\Pi_{\{1\}}^{k}(\{*\})$ in  our  sense. 
\end{example}

\begin{example}[$G$-Euclidean  Neighborhood  Retracts]
Slightly  more  general,  consider a compact lie  group $G$,  and  an  \emph{equivariant  euclidean  neighborhood  retract},  that  is,  a  compactly  generated  space  $X$  with  an  embedding  $X\to  V$  into  some finite  dimensional  representation  of  $G$, and  a  retraction $r:U\to  X$  of some  open  invariant subset $U\subset  V$. In  an  analogous  situation to  the  manifold case,    Tom Dieck \cite{tomdieckrepresentation}, p.188    constructs  a   collapse map,  called  the Lefschetz-Dold  index
 $$I_{G}(X):S^{V}\to S^{V}$$

This  defines  a  cocycle $(V,V,{\rm id},I_{G}(X)\mid_{V})\in \Pi_{G}^{0}(\{*\})$
\end{example}

\begin{example}[Parametrized Fixed Point  Situations]
Let  $B$ be  a metric space with  an  action of  a  compact  lie  group $G$. An  euclidean  neighborhood retract  over  $B$ is   an equivariant, locally  trivial  fibration  $p:E\to B$  endowed  with  a fiberwise   embedding   into  a  trivial  bundle   $E\to V\times B$.  A  parametrized  fixed  point  situation  is a diagram  of  the  form 
$$\xymatrix{E\times V \supset U \ar[rr]^{f}\ar[dr]_{p\circ{\rm proj}_{E}} &  & E\times V \ar[dl]^{p\circ {\rm proj}_{E}}  \\
 & B &} $$       
where $f$  is  a  compactly  fixed  map,  in the  sense  that  the  restriction of  $p$ to  the  fixed  point  set  ${\rm Fix}(f)=\{ (e, y)\in V \mid f(e,y)=(e,0)\}$ is  a  proper map  onto $B$.
 
Under  these  hypotheses  there  exists  a $G$-invariant  neighborhood  $N$ of ${\rm Fix}(f)$ which  is  relatively  compact  and  contained  in  $U$, as  well  as  a  positive  number  $\epsilon$  such  that  $\mid  {\rm proj}_{V} (fx)-{\rm proj}_{V} x\mid\geq \epsilon$ for all  $x\in \bar{N}-N$  for  some  $G$-invariant  metric  in $V$.
 There  exists a  map  $c:U\to S^{U}  $  which  maps  the  open  ball  of  radius  $\epsilon$   and  centre  $0$  homeomorphically  into   $U$.
 Consider  the  map $k(x)=(c({\rm proj}_{V}x)-{\rm proj}_{V}f(x), x)$  defined  on  $U$,   which  is proper. 
The    element  $(B\times  V, B\times  V, {\rm  id}, k)$ represents  a cocycle  in $\Pi_{G}^{0}(B)$.          

 Parametrized  equivariant  fixed  point  theory  was  developed  in detail  in \cite{ulrich}  in the  context  of  actions  of compact  lie  groups,  generalizing  the  foundational  work  of  Dold \cite{doldfixpunkt}.

\end{example}

\begin{example}[Nonlinear  Elliptic Operators with  Symmetry]
 Let  $\Omega\subset V$ be  a $G$-invariant, relatively  compact,  open domain with  smooth  boundary   inside  some  real  representation $V$  of  the  compact  lie group $G$.  Consider  the  H\"older  spaces  $C^{k+\mu} (\bar{\Omega}, V)=C^{k}(\bar{\Omega}, \mathbb{R})\otimes  V$ for  $\mu\in (0,1)$  and  an  elliptic  linear  operator  of  order m: 
$$P=\sum_{\mid \alpha\mid\leq m}  A_{\alpha} (X)D^{\alpha}$$
where  the  matrix-valued  functions  $A_{\alpha}: \bar{\Omega}\to Hom_{\mathbb{R}}(V,V)$ are  $G$-e\-qui\-va\-riant. 
Given  a  set of  well-posed, $G$-equivariant   boundary  conditions $B$, and  any  $\mu \in (0,1)$ Marzantowicz \cite{marzantowicz} constructs  a  $G$-equivariant, Fredholm  operator between  Banach  spaces   
$$P: C_{P, B}^{k+m+\mu}(V)= \{x\in  C^{k+m+\mu}(\bar{\Omega}, V) \mid B(x)=0\} \to C^{k+\mu}(\bar{\Omega}, V)$$
 defined  in a subspace $C_{P,B}^{k+m +\mu}(V)$  of  the  H\"older space satisfying  the  boundary  conditions.

In  an analogous  situation, considering  the  elliptic  operator  in the  Hilbert  space $H^{s, 2}(\Omega, \mathbb{R})$  of   all  functions  having formal  derivatives in $L^{2}$ up  to  a  certain order  $s$    gives  a  $G$-Hilbert  space  $H^{s, 2}(\bar{\Omega}, V)=H^{s}(\bar{\Omega}, \mathbb{R})\otimes V$,  as  well  as  a  Fredholm  operator  defined  on the  space  where  the  boundary  conditions  $B$  are  satisfied,  which  we  also  denote  by  $P_{B}: H^{s+m, 2}_{B}(\bar{\Omega}, V)\to H^{s, 2}_{B}(\bar{\Omega}, V)$   

 For  any   strongly  nonzero, nonlinear equivariant function defined  on the  radius one  disk  $\Phi: D_{1}(H^{s+m}_{B}(\bar{\Omega}), )\to   H^{s}(\bar{\Omega}, V)$,  the   cocycle 
$$(H^{s+m}_{B}(\bar{\Omega}, V), H^{s}(\bar{\Omega}, V),  P_{B}, \Phi)$$ 
represents  an  element  in  the  ``no-zeros  on the  boundary''-picture  of 
$$\Pi_{G}^{\rm {ker}(P_{B}) -{\rm coker}(P_{B})}(\{*\})$$
See \cite{marzantowicz}, definition 1.2  for  the   definition  of  an  explicit  isomorphism (the  equivariant Leray-Schauder  degree)   of  an specific   kind  of these  objects  to   the   equivariant  stems  of a  compact  Lie  group. This  is  the  equivariant  version  of  the   main  result  of  the  foundational  work  of  Albert  Schwartz,  \cite{schwartz}. See  also  \cite{ize}  for applications  of  related   constructions to Hopf  bifurcation problems. 
\end{example}

\subsection{A Bauer-Furuta  invariant  for  proper  actions on  4-manifolds}
 
We  illustrate now  an  example  where  our  constructions  in terms  of  perturbation of  Fredholm  morphisms appear in a natural  way.   This approach  was used  by   Stefan  Bauer \cite{bauer} and  Mikio  Furuta.

We  recall  briefly  that  in the  context  of smooth,  riemannian oriented manifolds, the  existence  of  a  ${\rm Spin}^{c}(4)$-structure  is  always  guaranteed.  This  amounts  to  a map  from a ${\rm Spin}^{c}$  principal  bundle $Q$ together  with  a  bundle  map $Q\to  P$  to  the frame bundle  of  the tangential bundle. The  identification of the  group ${\rm Spin}^{c}(4)$  with  the sugroup $\{u_{+},u_{-}\mid u_{-}, u_{+} \in {\rm U}(2), {\rm det} (u_{+})={\rm det}(u_{-})\}$  allows  to  define positive,  respectively, negative  spinor  bundles  $S^{+}, S^{-}= Q \underset{\rho^{+,-}}{\times}\mathbb{C}^{2}$,  where $\rho_{+,-}:  Spin^{c}\to {\rm U}(2)$  are  the  respective  projections.   Using  quaternionic  multiplication, it  is  posible  to  furnish $S:=S^{+}\oplus S^{-}$  with  the  structure  of  a  module  over  the Clifford  algebra of the  cotangential bundle  $T^{*}(X)\times  S^{+,-}\to S^{-,+}$. Clifford  identities  give  a linear  map $\rho: \Lambda^{2}\to {\rm End}_{C}(S^{+})$ whose  kernel  is the  bundle  of anti-selfdual $2$-forms  and  whose  image is  the  bundle of  trace  free skew  hermintian  endomorphisms. For  any  ${\rm spin}^{c}$-connection $A$, define  the associated  Dirac  operator $D$ 
 as the  composition $\Gamma(S^{+})\underset{\nabla_{A+a}}{\to} \Gamma(S^{+})\otimes \Lambda^{1}(T^{*}M)\overset{\gamma}{\to} \Gamma(S^{-})$, where  $\gamma$ denotes  Clifford  multiplication. 
 
The  \emph{monopole  map} $\tilde{\mu}$ is  defined  for four-tuples $(A, \phi, a, f)$  of  a  ${\rm Spin}^{c}$  connection $A$, a positive  spinor $\phi$,  a  $1$-form $a$  and  a  locally  constant  function $f$ on $M$
as 

\begin{multline*}
 $$\mu: {\rm Conn}\times \Gamma(S^{+})\oplus \Omega^{1}(M)\oplus H^{0}(M)\to \\
 {\rm Conn}\times \Gamma(S^{-})\oplus \Omega^{+}(M)\oplus \Omega^{0}(M)\oplus H^{1}(M) $$ 
\end{multline*}

$$(A,\phi,a,f)\mapsto (A, D_{A+a}\phi, F^{+}_{A+a}- \sigma(\phi), d^{*}(a)+f, a_{\rm  harm}) $$

where $\sigma$ is  the tracefree  endomorphism  $(-i)(\phi\otimes \phi^{*})- \frac{1}{2}\mid \phi\mid^{2}\dot{\rm id}$, and  $F^{+}$ denotes  the  self-dual  part  of  the  curvature. 
Given  a point  in $M$, the  based gauge  group $ \mathcal{G}_{x}$ is  the  kernel  of  the  evaluation  map  at  $x$.  ${\rm map}(X, \mathbb{S}^{1})\to  S^{1}$. The  subspace  $A+ker(d)$  is  invariant  under  the  free  action  of  the  based  gauge  group. The  quotient   is  isomorphic  to the  \emph{Picard torus}, $\mathfrak{Pic}(X)= H^{1}(X,\mathbb{R})/ H^{1}(X,\mathbb{Z})$. Let $\mathcal{A}$  and  $\mathcal{B}$ be  the  quotients
$$a+\ker{d}\times  \Gamma(S^{+})\oplus \Omega^{1}(X)\oplus H ^{0}(X,\mathbb{R})/G_{x}$$
respectively

$$ a+\ker{d}\times  \Gamma(S^{-})\oplus \Omega^{+}(X)\oplus \Omega^{0}(X,\mathbb{R})\oplus H^{1}(X, \mathbb{R})/G_{x}$$
the quotient  map  $\mu: \tilde{\mu}/G_{x}: \mathcal{A}\to \mathcal{B}$  has  by  definition  a presentation  as a  fibrewise  compact  perturbation  of  a  Fredholm  operator. It is  proper   after  a  result  of  Bauer  and  Furuta, \cite{bauer}, which  essentially  uses estimates  determined by  the  Weitzenb\"ock  formula. This  gives  rise  to  a  cocycle. $\mathcal{A}, \mathcal{B},D_{A}+ d +d^{*}, c=F^{+}_{A}+a \cdot  \phi + \sigma(\phi)$, where $\sigma$  is  the  selfdual trace free endomorphism $\phi \mapsto (-i\phi\otimes \phi^{*}- \frac{1}{2}\mid \phi\mid)$. 
 \begin{example}

Suppose  $G$ is  a  (possibly  noncompact)  Lie  group acting  properly  and  cocompactly  on the  smooth ${\rm Spin}^{c}$- manifold  $M$. assume  furthermore, that  the group  preserves  the  orientation and  by  means  of  isomorphisms  of  complex ${\rm Spin}^{c}$  structures,  and  respecting  the  ${\rm Spin}^{c}$-connection. As  the  action  is  proper,  it  can  be  assumed  that $G$  preserves  the  metric. Let $\mathbb{G}$ be  the  group  of pairs  $(\varphi,u)$, where $\varphi$ is a   $G$-equivariant  diffeomorphism which  preserves both  the metric and  the  orientation  and  $u: f^{*}(\sigma_{M})\to \sigma_{M}$ is  an isomorphism of  the ${\rm Spin}^{c}$  principal  bundle. In  particular, this  gives  a description of  $\mathbb{G}$ in  the middle  of  the following  exact  sequence 

$$1\to S^{1}\to \mathbb{G}\to G\to 1 $$

In this  situation, the  class of $\mu$  is  denoted  by   
$$m_{\mathbb{G}}(X,\sigma_{X})\in \Pi_{G}^{{\rm ind}(\lambda)}(\mathfrak{Pic}(X))$$
 and we call it   the  generalized  Bauer-Furuta invariant. The   restriction   map   

$${ {\rm  res}^{S^{1}}_{\mathbb{G}}:\Pi_{\mathbb{G}}^{{\rm ind}(l)}(\mathfrak{Pic}(X)) \to \Pi^{{\rm ind}(l)}_{S^{1}}(\mathfrak{Pic}(X))\cong  \Pi_{S^{1}}^{{{\rm ind}(l)}}(\mathfrak {Pic}(X))}  $$ 
which  will  be  constructed in the  following  section maps  $m_{\mathbb{G}}$  to the  $S^{1}$-equivariant  cohomotopical Bauer-Furuta  invariant defined  in \cite{bauer}. If, moreover,  vector  bundles  do  suffice  to  represent  the  equivariant  $KO$-theory  of  the  Picard  torus,  the element $\lambda$  can be  identified after  choice  of  an  invariant riemannian metric  on the  manifold with  the  difference  of the  complex  virtual  index  bundle  of  the  Dirac  operator  and  the  trivial  bundle  of  the  space   of  selfdual harmonic  two-forms.

\end{example}
\section{Cohomological  Properties  of Equivariant Cohomotopy for  proper  actions} \label{sectioncohomology}

We describe  now  an  additive  structure  in equivariant cohomotopy theory.

Let  $(E_{0}, F_{0} l_{0}, c_{0})$ and $(E_{1}, F_{1} l_{1}, c_{1})$
be  cocycles  in  the  equivariant  cohomotopy theory of  a   given
degree $[l]$.  

Let  us  suppose without  loss of  generality  that we
have  representatives   of  the  form $(E_{0}, F_{0}, l, c_{0})$  and
$(E_{0}, F_{0}, l, c_{1})$.  

Let  $X\times  \mathbb{R}\to X\times  \mathbb{R}$ be  the  trivial bundle  and . Denote by $S^{\mathbb{R}}$  the  one  point  compactification bundle. Define  the  pinching map  $\nabla:S^{\mathbb{R}}\to S^{\mathbb{R}}\vee S^{\mathbb{R}}\approx (S^{\infty},0)\vee (S^{-\infty}, 0)$   
$${
\nabla(x)=\begin{cases} ln(-x)\; \text{$(x\in -\infty, -1]$} \\  -ln(-x) \; \text{$x\in (-1, 0]$} \\ \infty \; \text{$x=0$} \\ ln(x)\; \text{$x\in(0,1]$} \\ -ln(x) \; \text{$x\in (1,\infty)$} \\ -\infty \; \text{$x=\infty$} \end{cases}}$$ 

The  sum  of  two  cocycles is  represented  by  the  cocycle $(E_{0}\oplus 
\mathbb{R}, F_{0}\oplus \mathbb{R}, l \oplus {\rm id}, c)$, where  $l\oplus {\rm id} + c: S^{E_{0}\oplus \mathbb{R} }\to S^{F_{0} \oplus  \mathbb{R}}$ is  given  as  the composition
 \begin{multline*}
$$ S^{E_{0}}\wedge_{X}  S^{\mathbb{R}} \overset{{\rm id} \wedge_{X} \nabla} {\to} S^{E_{0}}\wedge_{X} S^{\mathbb{R}}\vee_{X}S^{\mathbb{R}} \overset{\approx}{\to} \\ S^{E_{0}\oplus  \mathbb{R}} \vee_{X} S^{E_{0}\oplus \mathbb{R}} \overset{ (l\oplus {\rm id} + c_{0} )\wedge_{X} (l\oplus  {\rm  id} +c_{1})} { \to } S^{F_{0}\oplus \mathbb{R}}$$
 \end{multline*}

The   zero  element  is  represented  by a  cocycle $(E\oplus
 \mathbb{R} ,F\oplus \mathbb{R}, l, c)$ such that $c$ extends  to a   map sending  fibrewise $S^{E}$  to the point  at  infinity. .  
 
 The  inverse  of  an  element $(E, F, l, c)$  is  represented by  the
 element  $(E, F, l, -c)$.  
 We  have  the  following  result: 
 
\begin{proposition}     
The  operations  described  above  define  an  abelian  group  structure  in  equivariant  cohomotopy.
\end{proposition}

There  is  a  relative  version for  pairs $(X,A) $ of  proper  $G$-CW  complexes. An element in $\Pi^{[l]}_{G}(X,A)$ is  represented  by  a  compact  perturbation  of  a  fibrewise perturbation of a Fredholm  operator $l+c: E\to F$, which  extends  to the  one-point  compactification bundles being   constant  over the  subspace $A$, with  the  value at infinity. Notice  that  this  is  consistent  with  the  usual  identification of $X$  with the  pair $(X,\phi)$.

We   construct  a  multiplicative  structure  on  the  equivariant  cohomotopy theory: 
$$ \cup: \Pi_{G}^{[l_{1}]}(X,A_{1})\times \Pi_{G}^{[1_{2}]}(X,A_{2})\to \Pi_{G}^{[l_{1}+l_{2}]}(X  , A_{1}\cup A_{2}) $$

Consider  for  this  representing  elements $u_{i}=\varphi_{i}+c_{i}\in \Pi^{[l_{i}]}_{G}(X,A_{i})$ for  $i\in \{1,2\}$, where  $c_{i}$  is  a  compact  map  accepting  fibrewise  an  extension  to the  one-point  compactification,  constant   over  $A_{i}$  with the  value  at  infinity. $u_{1}\cup u_{2}$ is   the  cocycle  defined  as  $(E_{1}\oplus E_{2}, \varphi_{1}\oplus \varphi_{2}, C)$  where  the  map $C$ is  such  that  $C:(e_{1}, e_{2})\mapsto (c_{1}(e_{1}),c_{2}(e_{2}))$. Notice  that  this map  allows  an  extension  to the  one-point  compactification. 

We  investigate  now  the  cohomological behaviour  of  equivariant  cohomotopy: 

\begin{proposition}[Functoriality]
Let $f:(X, A)\to (Y, B)$ be  a $G$- map  between  finite  $G$-CW  complexes. Then $f$ induces  a  group  homomorphism 
$$\Pi_{G}^{[l]}(Y,B)\to \Pi_{G}^{f^{*}[l]}(X,A)$$
\end{proposition}
\begin{proof}
 Follows  from the naturality  of  the  pullback  construction  for  Hilbert  bundles, Fredholm morphisms  and  compact  perturbations.  
\end{proof}

\begin{proposition}
 Let $(X,A)$ be  a  proper $G$-CW pair. There exists   a  natural  sequence 
$$\Pi^{[l]}_{G}(X,A)\overset{\rho^{*}}{\to }\Pi^{[l]}_{G}(X) \overset{i^{*}}{\to} \Pi^{[l]}_{G}(A)$$
 which  is  exact  in the  middle, where  $\rho$ and  $i$  denote  the  inclusion of $A$ into $X$  and  $X$ into $(X,A)$, respectively.  
 \end{proposition}

\begin{proof}

That  $i^{*}\circ \rho^{*}=0$  is  clear form  the  definitions.  Let  now  $\varphi+c: E\to F$ be  a  cocycle  for  which $i^{*}([\varphi+c])$  is   compactly homotopic  to the  trivial  morphism  over  $A$.

 In view of  proposition \ref{proposition proper  stabilization}, we  can  choose a  representative (which  we  denote  by  the same  symbols) for  which  both  $E$ and $F$ are the  trivial  $G$-Hilbert  bundle $\mathcal{H}\times X \to X$ and  $c$  is  constant  over $A$  with  value  $\infty$. Using  proposition \ref{proposition linear  extension}, we  can  assume   up  to  equivalence that  the Fredholm  linear  operator extends  to  a  map $\tilde{\varphi}$ defined between bundles  defined over all points of  $X$.

 Suppose  that  there is  a  homotopy $h_{t}: i^{*}S^{E}\times  I \to i^{*}S^{F}$ defined  over  $A$  which  begins with  $i^{*}[\varphi+c]$ and  ends  with  a map  $\varphi+c$ which  sends  the   space $E$  to  the  base  point  at  infinity. As  $X-A$  is  built  up  out of  a  finite  number  of  equivariant  cells,  one  can  argue  inductively to  extend  $h$  to  a  map  $H: S^{E}\to S^{F}$, defined  on all $X$ such  that $h\mid_{A\times  I}=h$. $H$ determines  a homotopy  between  certain  element $\rho^{*}(\tilde{\varphi}+\tilde {c})$ defined  over  $X$ and $ \varphi+c$.

\end{proof}
\begin{proposition}
 Let $(X,B) $ be  a proper, finite   $G$-CW  pair obtained  as  the  pushout with  respect to  the   cellular map $(f,F):(X_{0}, A)\to (X, B)$ as  in the  following  diagram: 
$$ \xymatrix{A \ar[r]^{f} \ar[d]& B\ar[d] \\ X_{0} \ar[r]_{F}& X  }$$ 
then  the  map $(f,F)^{*}: \Pi_{G}^{[l]}(X,B)\to \Pi_{G}^{[l]}(X_{0}, A)$  induces  a natural  isomorphism. 

\end{proposition}

\begin{proof}
Let $(E, F\varphi+C) \in \Pi_{G}^{[l]}(X_{0}, A)$.  Due  to  propositon \ref{proposition linear extension}, it  is  possible  to  assume  that there  exists  a linear  morphism $\tilde{\varphi}:E\oplus E^{'}\to  F\oplus E^{'}$ defined  over  $X$  such  that  $F^{*}(\tilde{\varphi})$  and  $\varphi$  differ  by  addition  of  an  unitary  cocycle. As $c\mid_{A}=\infty$, it  is posible  to  extend  $c$ to a  map  $\tilde{c}$  defined  on  $X$ for  which  $\tilde{c}\mid_{B}=\infty$. Then $(F^{*}, f^{*}): \Pi_{G}^{*}(X, B)\to \Pi_{G}^{*}(X_{0}, A)$ sends $\tilde{\varphi}+\tilde{c}$ to $\varphi+c$. This  proves  surjectivity. To  prove  injectivity,  recall  that  if  $\varphi +h_{t}:  E\times  I \to E$ is a  nullhomotopy   starting  with  $(F^{*}, f^{*})(\varphi+ c)$,  ending  with  the constant  $\infty$.  As  before,  as  $c$  is  constant  over   $B$ and  $(F^{*}, f^{*})(c)$ is  trivial  over  $A$, it  is  possible  to  extend  the  map   $h_{t}$ to   a  homotopy  $\tilde{h}_{t}$ defined  over $X$ which  is  trivial  over  $B$ and  which  begins  with  $\varphi+ c$,  ends  with  a  constant  map. This  shows  that  the  map  is  injective. 
\end{proof}

\begin{proposition}[Homotopy invariance]
Let  $f_{0}, f_{1}:(X,A)\to (Y,B)$ be   two $G$-maps of  pairs  of  proper  $G$-CW  complexes. If  they  are  homotopic,  then  $\Pi^{[l]}_{G}(f_{1})=\Pi_{G}^{[l]}(f_{2})$.    
 \end{proposition}
\begin{proof}
 In  view  of  the  naturality  of  the  construction,  this  amounts  to  prove that  $\Pi_{G}^{[l]}(h)= {\rm  id}$ for  the  map  $h: (X,A)\times  I \to (X, A)\times  I$ given by $(x,t)\mapsto (x,0)$. Let  $( E, F, l, c)$  be  a  nonlinear  cocycle representing  an element  in  ${\rm  Fix}_{G}^{*}(X,A\times  I)$. In the  notation  of  proposition  \ref{proposition Hilberthomotopy}, there  exist unitary  morphisms $u: E\to h^{*}(E)$, $v:F\to h^{*}(F)$  covering  the  identity $X\times  I\to X\times I$ such  that  the  restrictions  to $E_{0}$, $F_{0}$  are  the  respective  identities. Note  that  the  composition $f= h^{*}(E)\overset{u^{-1}}{\to} E \overset{\varphi+c}{\to}  F \overset{v^{-1}}{\to} h^{*}(F)$ is  homotopic   to $h^{*}(\varphi+ c)$  relative  to  $h^{*}(\varphi)$. After checking  out the  conditions  for  the  definition  of  $\Pi_{G}^{[l]}(X)$, one  has  that   the  equivalence classes  $\Pi_{G}^{[l]}(h)(\varphi+c)=[f] = [\varphi+c]$ agree.
\end{proof}

\begin{definition}
 Let  $[l]$ be  a class  in $KO_{0}^{G}(X)$. The  group $\Pi_{G}^{[l+1]}(X)$  is  defined  to  be   the  group  of  equivalence  classes  represented  by  cocycles   defined  over  $X$, which   admit  a  representation by  a  cocycle  $(E\oplus  \mathbb{R}, E , l\circ {\rm proj}_{E}, c)$, where  ${\rm proj}_{E}$  denotes  the  projection  onto  the  subspace  $E$. 
\end{definition}

We  construct  a  suspension  isomorphism  

$$\sigma_{[l]}^{X,A}:  \Pi_{G}^{[l]}(X, A) \to \Pi_{G}^{[l]+1}((X, A)\times (I, \{0,1\}))$$
 
Given  $l+c : E\to  F \in  \Pi_{G}^{[l]}((X,A)$, form  the  bundle $E^{'}= E\oplus \mathbb{R}$, denote  by  $p$  the  fibrewise  projection  on  $  E\oplus  \mathbb{R} \to  E$  and   define the  map  $\sigma_{[l]}(l +c): DE^{'}\times I\to E^{'}\times  I$  defined  as $(e,t)\mapsto  p \circ l + (\log(t)- (\log(-t)))(c(p(v))$. By  this  means,  we  obtain a  cocycle $(E\oplus \mathbb{R}, E,p\circ l,C)$    which  extends  to  the  fibrewise one-point compactifications, beeing  trivial  on  the required  subspace.   Given  an  element $\varphi+c \in \Pi_{G}^{[l+1]}((X,A)\times  I,\{0,1\})$, consider  a unitary, fibrewise  linear  cocycle  $u: E\to E_{0} \times  I$ $v: F\to  F_{0}$  covering  the  identity $X\times  I\to  X\times  I$,  which  restricted  over  the  subspace $ A\times  I \cup X \times \{0,1\}$  is  the  identity map. The  map constructed  as  $E_{0}\times  I \overset{u^{-1}}{\to} E\overset{\varphi+c}{\to} F\overset{v}{\to} F_{0}\times I$ determines  an  inverse  for  $\sigma^{(X,A)}_{[l]}$.  

A  coboundary map is  defined  as the  composition
\begin{multline*}
$$ \Pi_{G}^{[l]}(A)\overset{\sigma_{G}^{n}(A)}{\longrightarrow} \Pi_{G}^{[l+1]}( A\times
    I, A\times  \{0,1\}) \\ \overset{\Pi_{G}^{[l+1]}(i_{1})^{-1}
    }{\longrightarrow } \Pi_{G}^{[l+1]} (X\cup_{A\times\{0\} } A\times
    I, X\coprod A\times  \{1\} ) \\ \overset{\Pi_{G}^{[l+1]}(i_{2}) }
    {\longrightarrow} \Pi_{G}^{[l+1]}(X\cup_{A\times \{0\}} A\times I,
    A\times \{1\}) \\ \overset{\Pi_{G}^{[l+1]}({\rm
    pr}_{1})^{-1}} {\longrightarrow} \Pi_{G}^{[l+1]}(X, A)   $$      
\end{multline*}
 Where  the  map $\Pi^{n+1}_{G}(i_{1})$ is  bijective  by  excision and
 $\Pi_{G}^{n+1}({\rm pr}_{1})$  is  bijective  because  of  homotopy
 invariance.   

We  analize  now  induction  and  restriction  structures. In order  to define  the  induction  structure  in equivariant  cohomotopy,   we  restrict  ourselves  to  the  case  where  $G$  is  a  Lie  group. 

\begin{proposition}
Let  $\alpha: H\to  G$  be a proper Lie group  homomorphism . Then  there  exists a group  homomorphisn 
$$\Pi_{G}^{[l]}({\rm ind}_{\alpha}(X,A)) \to \Pi _{H}^{[l]}(X,A)$$  
satisfying 
\begin{enumerate}
 \item{Bijectivity. If  $\ker (\alpha)$  acts  freely  on  $(X, A)$, then the  map is  an  isomorphism.}
 \item{Compatibility  with  the  boundary  homomorphisms. $\delta^{[l]}_{H}\circ {\rm Ind}_{\alpha}=\ind_{\alpha}\circ\delta^{[l]}_{G}$.  }
 \item{Functoriality. If $\beta: G\to K$  is a  group  homomorphism, then  the  diagram commutes: }
 \item{Compatibility with conjugation.For  any  $g\in G$ , the  homomorphism 

$${\rm ind}_{c_(g):G\to G}\Pi^{[l]}_{G}(X,A)\to \Pi^{[l]}_{G}(\rm ind)_{c(g):G\to G}(X,A))$$

agrees  with  the  map  $\Pi^{[l]}_{G}(f_{2})$, where  $f_{2}: (X,A)\to {\rm ind}_{c(g):G\to G}$ sends $x$  to  $(1,g^{-1}x)$ and $c(g)$  is  the  conjugation  isomorphism in $G$ associated  to $g$} 
\end{enumerate}

\end{proposition}
\begin{proof}
\begin{enumerate}
 \item{Let $\varphi+ c: E \to E$ be  a  compact perturbation over  the  space $({\rm  ind}_{\alpha}X,A)$. The  map  ${\rm i} : X\to {\rm ind}_{\alpha}(X)$ $(x \mapsto(1_{G},X)$  induces  a   group  homomorphism 
$$\Pi_{G}^{*}({\rm ind}_{\alpha} X, A) \to \Pi_{H}^{[l]}(X, A)$$. 
An  inverse is  given  by  the  map which  associates  to  a  linear  cocycle $(E, F, \varphi)$ the   cocycle  $(E/ H, F/H, \varphi /H)$. It  is  easy  to  show  that  this   is  the  case  for  the  perturbation   and  that  this  still  satisfies   the  boundedness  condition. We  point  out  that  we state this  fact  for proper  actions of  Lie  groups, see remark below. }
\item{Follows  from  the  naturality of the induced  bundle  construcctions.}
\item{Follows  from  the functoriality    of  the induced  vector bundle  construction. }
\item{Compatibility  with conjugation. Follows from  element  chasing  in the  diagram 
$$\xymatrix{\Pi_{G}^{[l]}({\rm ind}_{c(g)}  (X,A)) \ar[r]^{{\rm ind}_{c(g)}}&  \Pi_{G}^{[l]}(X, A)  \\  \Pi_{G}^{[l]}({\rm ind }_{c_(g)} (X,A))  \ar[u]^{=}\ar[ur]_{\Pi_{G}^{[l]}(f_{2})} & }$$
where $f_{2}: (X, A) \to {\rm ind}_{c(g)} (X,A)$  is  given by \\
$x \mapsto(1,g^{-1}x)$ and $c(g)(g^{-1})=gg^{'}g^{-1}.$}
\end{enumerate}
 \end{proof}

\begin{remark}
 N.C. Phillips proves  in  \cite{phillipsktheory}, corollary 8.5 , p. 131 a   more  general  result for  equivariant  $K$-theory, allowing  proper actions  of locally compact  groups as  input, instead  of  only  Lie  groups. The  adaptation  of  these  methods  is  certainly plausible. The  main problematic point  is  the  induction  structure. We  cannot  guarantee  the  local  triviality  of  bundles  $E/H\to X/H$ unless $H$ is  Lie, see \cite{palais}.   
\end{remark}

\section{The  parametrized  Schwartz-Index}\label{sectionindex}

\subsection{Review  of  Equivariant  Cohomotopy  for  Proper  Actions  of  a Discrete Group}\label{sectionlueck}

Several  approaches  have  been  proposed  towards  the  definition  of   equivariant  cohomotopy  theory  for  proper  actions. L\"uck  \cite{lueckeqstable} uses finite  dimensional  bundles. This allows  to  deal  with the  difficulties appearing  in the   case  where  a  discrete group  acts on   a  finite  $G$-CW  complex.  We  briefly  recal  this  approach

Fix an equivariant, proper  $G$-CW complex. Form the category

${\rm SPHB}^{G}(X)$ having  as  objects  the $G$-sphere bundles over $X$. A
morphism from $\xi:E\to X$ to $\mu:F\to X$ is a bundle map $S^{\xi}\to
S^ {\mu}$ covering  the identity in $X$, which preserves fiberwise  the
basic points. A  homotopy between the  morphisms $u_{0}$, $u_{1}$ is a
$G$-bundle map $h:S^{\xi}\times [0,1]\to S^ {\mu}$ from the bundle
$S^{\xi}\times [0,1]\to [0,1]\times X$  to the  bundle $S^{\mu}$
covering   the  projection $X\times [0,1]\longrightarrow X$ and
preserving the base points on  every  fiber such that its  restriction
to $X\times \{i\}$ is $u_{i}$ for $i=0,1$.
Let $\underline{\mathbb{R}^{n}}$ be  the trivial vector bundle  over
$X$, which is  furnished with the trivial action of  $G$. Two
morphisms  of  the  form
$$S^{\xi_{i}\oplus \underline{\mathbb{R}^{k_{i}}}}\to S^{\xi_{i}\oplus
  \underline{\mathbb{R}^{k_{i}+n}}}$$

are said to be equivalent  if there are objects $\mu_{i}$ in
  ${\rm SPHB}^{G}(X)$ and an isomorphism of vector bundles $\nu: \mu_{0}\oplus
  \xi_{0}\cong \mu_{1}\oplus \xi_{1}$  such that  the following
  diagram of morphisms  in ${\rm SPHB}^{G}(X)$ commutes up to homotopy
$$ \xymatrix{ S^{\mu_{0}\oplus\underline{\mathbb{R}^{k_{1}}}} \wedge_{X}
      S^{\xi_{0}\oplus\underline{\mathbb{R}^{k_{0}}}}
      \ar[rr]^{\rm{id} \wedge_{X} u_{0}} \ar[d]_{\sigma_{1}} &  &
      S^{\mu_{0}\oplus\underline{\mathbb{R}^{k_{1}}}} \wedge_{X}
      S^{\xi_{0}\oplus \underline{\mathbb{R}^{k_{0}+n}}}
      \ar[d]^{\sigma_{2}} \\    S^{\mu_{0}\oplus \xi_{0}\oplus
      \underline{\mathbb{R}^{k_{0}+k_{1}} }} \ar[d]_{S^{\nu\oplus
      \rm{id}}} &  &  S^{\mu_{0}\oplus \xi_{0}\oplus
      \underline{\mathbb{R}^{k_{0}+k_{1}+n} }}\ar[d]^{S^{\nu \oplus
      \rm{id}}}  \\      S^{\mu_{1}\oplus \xi_{1}\oplus
      \underline{\mathbb{R}^{k_{0}+k_{1}} }}
      \ar[d]_{\sigma_{3}}&  & S^{\mu_{1}\oplus \xi_{1}\oplus
      \underline{\mathbb{R}^{k_{0}+k_{1}+n} }} \ar[d]^{\sigma_{4}} \\
      S^{\mu_{1}\oplus\underline{\mathbb{R}^{k_{0}}}} \wedge_{X}
      S^{\xi_{1}\oplus\underline{\mathbb{R}^{k_{1}}}} \ar[rr]_{{\rm id}
      \wedge_{X} u_{1}} &  & S^{\mu_{1}\oplus\underline{
      \mathbb{R}^{k_{0}}}}\wedge_{X}
      S^{\xi_{1}\oplus\underline{\mathbb{R}^{k_{1}+n}}}}$$

where the isomorphisms $\sigma_{i}$ are determined by  the fiberwise
      defined  homeomorphism $S^{V\oplus W}\approx S^{V}\wedge S^{W}$
      and  the  associativity  of  smash products,  which  holds  for
      every  pair  of representations $V$,$W$. 
We  recall now  W. L\"uck's definition of  equivariant  cohomotopy: 

\begin{definition}
Let $X$ be a $G$-CW complex, where $G$ is a  discrete  group and  $X$  is  finite. We define its $n$-th G-equivariant stable
cohomotopy  group  $\pi_{G}^{n}(X)$ as  the set  of homotopy  clases of equivalence
classes of morphisms
$u:S^{\xi\oplus\underline{\mathbb{R}^{k}}}\to
  S^{\xi\oplus\underline{\mathbb{R}^{k+n}}}$ under  the  above
  mentioned relation. For a $G$-CW pair, $(X,A)$ we   define $\Pi_{G}^
  {n}(X,A)$ as  the equivalence classes of  morphisms  which  are trivial over
  $A$, i.e. those  which are  given by a  representative
  $u:S^{\xi\oplus\underline{\mathbb{R}^{k}}} \to
  S^{\xi\oplus\underline{\mathbb{R}^{k+n}}}$  which satisfies  that
  over
  every  point  $a\in A$, the  map $u_{a}:S^{\xi_{a}\oplus
    \mathbb{R}^{k}}\to  S^{\xi_{a}\oplus\underline{
    \mathbb{R}^{k+n}}}$ is constant with value the  base point.
For a pair  of bundle morphisms $u:S^{\xi\oplus\underline{\mathbb{R}^{k}}}\to
  S^{\xi\oplus\underline{\mathbb{R}^{k+n}}}$,
  $v:S^{\xi^{'}\oplus\underline{\mathbb{R}^{k}}}\to
  S^{\xi^{'}\oplus\underline{\mathbb{R}^{k+n}}}$, the  sum  is
  defined as  the  homotopy  class of the  morphism 
\begin{multline*}
$$u:S^{\xi\oplus \xi^{'}\oplus \underline{\mathbb{R}^{k}}}\wedge_{X}S^{\mathbb{R}}
  \overset{{\rm id} \wedge_{X} \nabla }{\to} S^{\xi\oplus
  \xi^{'}\oplus \underline{\mathbb{R}^{k}}}\wedge_{X}(S^{\mathbb{R}}\vee_{X}
  S^{\mathbb{R}} ) \overset{\sigma_{3}}{\to} \\ ( S^{\xi \oplus
  \underline{\mathbb{R}^{k} }}\wedge_{X} S^{\mathbb{R}} ) \vee_{X}(
  S^{\xi^{'} \oplus\underline{\mathbb{R}^{k} }}  \wedge_{X}
  S^{\mathbb{R}}) \overset{(u\wedge_{X} id)\vee_{X}
  (v\wedge_{X}id)}{\to} \\  S^{\xi\oplus
    \xi^{'}\oplus \underline{\mathbb{R}^{k+n}}}\wedge_{X}S^{\mathbb{R}}$$ 
\end{multline*}

where  $\sigma_{3}$ is the canonical isomorphism  given
by the  fiberwise distributivity  and  associativity  isomorphisms
and $\nabla$ denotes  the  pinching map 
$S^{\mathbb{R}}\to S^{\mathbb{R}}\vee S^{\mathbb{R}}$. The
    relative  version for  elements  lying  in the  group  of a
    pair, $\pi^{n}_{G}(X,A)$ translates word by  word  when one sets all
    sphere  bundles and  morphisms  to be trivial over $A$.      
\end{definition}

Our  approach   extends  the  notions proposed  first  by  L\"uck   and  solves  its  principal problem, namely:   the  the lack  of  finite dimensional  $G$-vector  bundles  to  represent  excisive  $G$-cohomology theories. The  crucial  result  in this  is  the  construction  of  an  index theory  in  the  context of  parametrized  nonlinear  analysis. Previous  versions  of  this  index  theory  (restricted  to the  unparametrized and   $S^{1}$-equivariant  case)   were  constructed  in \cite{bauer},  from  where  we  adopt the  crucial ideas.
 
We  prepare  its  definition  with  the  following   facts, which   are parametrized  versions  of  the  discussion in  page  5  and   Lemma  2.5 in \cite{bauer}. We  also  use the  notation from that  article. In  this  section, all  groups  are  discrete  and  all $G$-CW  complexes  are  finite.

\begin{definition}
Let    $l: E \to E$ be a  Fredholm  morphism over  a  proper  $G$-CW  complex  $X$,  equipped  with  some  invariant  riemannian metric. For  any  finite  dimensional $G$-vector  subbundle $\xi$ embedded  as  direct  summand, the  orthogonal spherical retraction associated  to $\xi$,   $\rho_{\xi}:S^{E}-S^{\xi^{\perp}}\to S^{\xi}$  is  defined  to  be  the   map given  on  every  fiber $x$ as $w\mapsto \vert w\vert {\rm Proj}_{\xi}(w)$   
\end{definition}

Using  appropiate  rescaling  homotopies  we  have: 

\begin{proposition}
The    homotopy  type  of an orthogonal  spherical  retraction  does  not  depend   on the  choice  of  a  riemannian  metric.  
\end{proposition}

The  following  is  a  parametrized  version  of  lemma 2.5  in \cite{bauer}, compare  also Corollary  3.11 in \cite{telemanokonek}: 

\begin{proposition}\label{proposition nonlinearapproximation}
Let $G$  be  a  discrete  group  acting  on a  finite  $G$-Cw  complex. Let  $f=l+c: E\to E $  be  a fibrewise compact  perturbation  of the  linear Fredholm morphism $l$ defined  in a $G$- Hilbert bundle $E$  over  the  proper $G$-space $X$.  Then,  there  exist finite  dimensional $G$-vector  bundles $\xi \subset E$ such  that
\begin{itemize}
 \item {$\xi$  spans together  with  the image  of  $l$  the $G$-Hilbert bundle  $E$.}
 \item {Given a  fibrewise  inclusion  as  orthogonal  summand  in a  finite  dimensional  vector  bundle  $\xi\to \tau$,  $\xi\oplus \zeta = \tau$, the  restricted  map $f:S^{l^{-1}(\xi)}\to S^{\xi}$  sends the  unit  sphere  $S^{\xi^{\perp} }$ to  the  orthogonal  complement  of  $\xi$.}
 \item {The  maps  $\rho_{\tau} f\mid_{l^{-1}(\tau)}$  and $id_{\zeta}\wedge_{X} \rho_{\xi}$  are $G$-homotopic  over $X$. }
\end{itemize}
\end{proposition}
\begin{proof}
 
\begin{enumerate}
  \item {Denote  by $D$ the  unit  ball bundle in $E$. Due  to  the  boundedness  condition,  the  map $f^{-1}(D_{x})$ is  bounded over  any  point  $x\in X$. Hence, the  closure  $C_{x}$ of  its  image  under  the  compact  map  $c$  is  fibrewise  compact. It  follows that  there exist continuous  sections  $\{ \phi_{i}: X\to E\}_{i=1}^{n}$  such  that  $C_{x} \subset B_{\epsilon}(\phi_{i}(x)$.  One  can  furthermore  assume  that  $\{phi_{i}(x)\}$  form a  linearly  independent set and the linear  subspace  spanned  by  them  intersects  trivially the orthogonal complement of  the image,   $l(E)^{\perp}$ on each  fiber.  The  subbundle  given  as  the  direct  sum $l(E)^{\perp}\oplus \langle \phi_{1}, \ldots,\phi_{n}(x)\rangle$ satisfies  the  required  properties. }
  \item{If $w_{x}\in  S^{\tau^{\perp}}_{x}$  is  in the image  of  $f\mid_{S^{l^{-1}(\tau)}}$, then  $f^{-1}(w_{x}))\cap l^{-1}(\tau)$ will be  mapped  under $f\mid_{l^{-1}}$  to a  subspace  of  $\tau_{x}+ C_{x}$. So, $w_{x}$ will be contained  in $S^{\tau^{\perp}}\cap \tau_{x}+C_{x}$, which  is  not  possible, because  the  distance  between  these  subspaces  is  greater  than  $1-\epsilon_{x}> \frac{3}{4}$.}
  \item{In  view  of  the  slice  theorem and  the  local  triviality  of  the  $G$-Hilbert  bundles  involved,  we  can  cover  the  space  $X$  with   invariant  neighborhoods  for  which  there  is a  map  $U_{x}\to G/H$, and  the  bundle  over $U_{x}$  is  the  pullback   of  the  bundle  $G\underset{H}{\times} E_{x}\to  G/H$ ,  where  $E_{x}$  is  some  strong,  norm  continuous representation  of  $H$  in a  Hilbert  space. Hence,  we  can  restrict  ourselves  to  bundles  over  an  orbit.  In the  notation  of the  previous  part, there is a  retraction $\rho_{\tau_{x}}: S^{\tau_{x}}\to S^{E_{x}}- S^{\tau_{x}^{\perp}}$.  We consider  the  isomorphism $l^{-1}(\tau_{x})\cong \zeta_{x}\oplus l^{-1}(\xi_{x})$  given  by $w_{x}^{'}\mapsto (l\circ (1-{\rm pr}_{l^{-1}(\xi_{x})}w^{'}_{x}, {\rm pr}_{l^{-1}(\xi_{x})} w_{x}^{'})$. We  claim  that  after  this  isomorphism, the  maps  ${\rm id}_{\zeta_{x}}\smash \rho_{\xi_{x}} (f\mid_{l^{-1}(\xi_{x})})$ and  $f\mid_{S^{l^{-1}(\tau_{x})}}$ are  homotopic. Consider  for  this   a  ball $D\subset E$ which  contains  the  inverse  image  $f^{-1}(D_{1}(0))$ of  the unitary  ball We  define  the  homotopy  $h: D\times  I\to S^{E}-S^{\tau^{\perp}}$ as follows

$${h(w_{x},t)=\begin{cases} l+[(1-3t){\rm id}_{E_{x}}+(3t){\rm pr}_{\xi_{x}}]\circ c &\text{$t\in [0, \frac{1}{3}]$} \\ l+{\rm pr}_{\xi_{x}}\circ c [ (2-3t){\rm id}_{l^{-1}(\xi_{x})} + (3t-1){\rm pr}_{l^{-1}(\xi_{x})} ]  &\text{$t\in  [\frac{1}{3}, \frac{2}{3}]$} \\{\rm pr}_{\zeta_{x}}\circ l + [(3-3t){\rm pr}_{\xi_{x}}+ (3t-2)\rho_{\zeta_{x}} \circ (l+c)\circ{\rm pr}_{l^{-1}(\zeta_{x})}] &\text{$t\in [\frac{2}{3}, 1]$} 
\end{cases}}$$
Since $S^{E}_{x}-D\cap \tau^{\perp}_{x}$  is  contractible, the  homotopy  above  can  be  extended to  a  homotopy $S^{l^{-1}(\tau)}\times  I\to  S^{E}-S^{\tau^{\perp}}\simeq S^{\tau}$,   as needed.   
 }
\end{enumerate}
\end{proof}

\begin{definition}[Parametrized  Schwartz  index]
 Let $G$ be  a   discrete  group  acting  on  the  proper  $G$-CW complex  $X$.  Denote  by  $(E, F, l, k)$  a  non-linear  cocycle for  the  equivariant  cohomotopy -theory  over  the  proper  $G$-space $X$,  where  $l$  is  a  linear  morphism  whose  index  bundle  is  trivial of dimension p. Let  $\xi$  be  a  finite dimensional  $G$-vector  bundle  as  constructed  in  proposition \ref{proposition nonlinearapproximation}.  The  parametrized  Schwartz  index  is  the  class of the element. 
$$ [p_{\xi}(l+k)\mid_{l^{-1}(\xi)}]: S^{l^{-1}(\xi)}\to S^{\xi}\in \pi_{G}^{{\rm ind}(l)}(X)$$ 
in  the  equivariant  cohomotopy  group  as,  introduced  by L\"uck in \cite{lueckeqstable}. This  construction  does  not  depend  of  the  choice  of  the  finite  dimensional  vector  bundle as  a consequence  of  part  3  of  proposition \ref{proposition nonlinearapproximation}.
\end{definition}

We  are  now  able  to  state  our  main  result: 

\begin{theorem}\label{theorem schwartzindex}
Let  $G$ be  a  discrete  group  acting  on a  $G$-CW  pair  $(X,A)$. Denote  by  $[l]$ the  $KO_{G}$-theoretical  class of  a  Fredholm  morphism  whose  fibrewise  index  is  a  trivial virtual  vector  bundle  of  dimension  $p$. The  parametrized  Schwartz  index gives  an  isomorphism
$$\Pi_{G}^{[l]}(X, A)\to  \pi_{G}^{p}(X,A) $$
\end{theorem}
\begin{proof}
We  construct  a  natural  inverse  map.  Using  the  stability  conditions, we can add finite  dimensional  vector  bundles and   assume  that  the  morphism  $u: S^{\xi\oplus \mathbb{R}^{n+p}}\to S^{\xi \oplus \mathbb{R}^{n}}$    is homotopic  to  a  map  such  for  which on  every  fiber,   the  only   preimage  of  the  basis  point  $\infty$  is  $\infty$. We  denote  by $c:\Gamma(\xi\oplus \mathbb{R}^{n+p}) \to \Gamma (\xi\oplus  \mathbb{R}^{n})$ the  (possibly  nonlinear) map  obtained  by  restricting $u$  fibrewise to the  complement  of  the  point at  infinity.     

 Consider  the   projection  operator  $P_{\xi\oplus \mathbb{R}^{n}}: \Gamma(\xi\oplus \mathbb{R}^{n+p}) \to \Gamma (\xi\oplus  \mathbb{R}^{n})$.  Let  $\mathcal{H}$ be  the stable   $G$-hilbert  space   of  proposition \ref{proposition proper stabilization}. Recall  that  $\Gamma(\eta)\oplus\Gamma(X \times \mathcal{H})\cong \Gamma(X\times  \mathcal{H})$ for all  $G$-Hilbert  bundles $\eta$, due  to  the  proper  stabilization theorem.   The  map  of  $C(X)$- $G$- Hilbert modules  
$$P_{\xi\oplus\mathbb{R}^{n}}\otimes_{C(X)}{\rm id}:   \Gamma(\xi\oplus \mathbb{R}^{n+p})\otimes \Gamma(X\times  \mathcal{H}) \to \Gamma(\xi\oplus \mathbb{R}^{n})\otimes \Gamma(\mathcal{H}\times  X)\cong \Gamma(X\times  \mathcal{H})$$
  determines   up  to  precomposition  with   the  proper stability isomorphism  a  Fredholm morphism $l:X\times \mathcal{H}\to X\times  \mathcal{H}$, which  represents  a  class   for  which   the  index  bundle  is  trivial  of  virtual  dimension $p$.
We  denote  by  $C$ the  proper, compact, nonlinear   map  defined  on  every  fiber  by 
$$\Gamma(X\times \mathcal{H})\overset{\cong}{\to}           \Gamma(X\times \mathcal{H}) \oplus  \Gamma(\xi\oplus \mathbb{R}^{n+p}) \overset{0\oplus  c}{\to}             \Gamma(X\times \mathcal{H})\oplus  \Gamma(\xi\oplus \mathbb{R}^{n})  \to     \Gamma(X\times \mathcal{H})$$
where  the  first  map  is   given  by  the  proper  stabilization isomorphism,  and  the  second  one  is  the inclusion  in the first  factor  of  the  direct  sum  followed  by the  proper  stabilization  isomorphism. 
The cocycle  associated  to $u$  is   $(X\times \mathcal{H}, X\times \mathcal{H}, l, C)$. 
. 
\end{proof}

\begin{remark}
In the case  of  compact Lie groups,  the  stability  condition \ref{condition triviality} allows  to  suppose  that  the index bundle $\ker - {\rm coker}$ has  the  form  of  a  trivial  bundle  $X \times  V$,  because  the  classifying  space has  the   equivariant  homotopy  type  of  a  point. The equivariant  Schwartz  index  identifies this  with  the  usual  definition  for  equivariant  cohomotopy  groups for  finite  $G$-complexes,  in  such a  way  that our  proof  for  discrete  groups  is  formally  the  same, with  the  classical definition above instead  of that  of  L\"uck. By specifying to  the  trivial group and  the  one-point space, this  theorem can be  traced to the main  result  in \cite{schwartz}, and in  \cite{bauer} one  finds  an $S^{1}$-equivariant  version  which  basically deals with all  problems  on  compact  Lie groups. In view  of  this correspondence, we use  in the  rest of  this  note  the  notation $\Pi$  for cohomotopy or  its  equivariant  generalizations.  
\end{remark}

\section{A Burnside  Ring  for  Lie  Groups}\label{sectionburnside}
We now  define a Burnside  ring  in  operator theoretical  terms for  non compact  Lie  groups. We  first  recall  the  definition for  compact  Lie  groups, which   was  first  indtroduced  by Tom Dieck  in \cite{tomdieckburnside}. 

\begin{definition}
Let $G$  be  a  compact  Lie  group. Consider the folowing equivalence
relation on the  collection of  finite $G$-$CW$ complexes. $X\sim $Y if  and
only if  for  all  $H\subset G$, the  spaces $X^{H}$ and $X^{H}$ have
the  same  Euler  characteristic. Let  $A(G)$ be  the set of
equivalence  classes. Disjoint  union and  cartesian product  of
complexes are  compatible  with this  equivalence  relation and
induce  composition laws on $A(G)$. It  is  easy  to  verify  that
$A(G)$ together  with  these composition laws  is  a commutative  ring
with identity. The zero  element  is  represented  by a  complex $X$
such that  the  Euler  characteristic $\chi(X^{H})$ is  zero  for  each
$H\subset G$. If $K$ is  a  space  with  trivial $G$-action and
$\chi(K)=-1$, then $X\times  K$ represents  the  additive inverse of
$X$ in $A(G)$. 

\end{definition}
 We  collect  some information about  the  algebraic  structure of  the  Burnside  ring  in the  followiing  results, which  have been  published  by  Tom Dieck  in \cite{tomdieck}, pages 240 250 and  256, respectively.

 \begin{proposition}\label{technical proposition tomdieck}
\begin{enumerate}
\item{As  abelian  group $A(G)$ is  the  free  abelian group on $G/H$, where
$H \in \Phi(G)$ and  $\Phi(G)$ denotes  the  space of conjugacy classes
of  subgroups such that $N(H,G)/H$ is  finite, where  $N(H,G)$ denotes
the normalizer of  $H$  in $G$.}
\item{ There  is  a character  map ${\rm char}_{G}: A(G)\to
{\rm Map}(\Phi(G), \mathbb{Z})$, where   $\Phi(G)$, the  space  of
closed  subgroups of $G$ carries  the Hausdorff  metric (in particular
it is a  compact   Hausdorff space). And  ${\rm char }_{G}(X) $ is  defined
 by   $ H \mapsto X^{H}$.}
\end{enumerate}  
\end{proposition}

\begin{proposition}
By  means  of  the  character  map, the  elements of  the burnside  ring  can  be identified  with  sums 

$$\underset{K}{\sum} n(H,K) \chi(X^{K}) \cong 0 \; {\rm mod} \; \mid
NH/H \mid   \quad (*)$$

where the sum is  over  conjugacy  classes  $(K)$ such that $H$ is
normal  in $K$,  $K/H \subset N_{H,G}/H$ is  cyclic, the integer  numbers
$n(H,K)$ are  defined  to  be 
$$n(h,K)=\mid {\rm Gen} (K/H) \mid \; \mid W_{H,G} /N_{W_{K/H},
  W_{H,G}}\mid$$ 
 and  ${\rm Gen}(Z)$ denotes  the  cardinality  of the  generators of
  the finite cyclic group $Z$. In particular, the  rationalized  Burnside  ring $A(G)\otimes \mathbb{Q}$ can be  identified  with  the  ring  of continuous rational  functions  defined  on $\Phi(G)$
\end{proposition}

\begin{theorem}
Let  $G$  be  a  compact  Lie  group. There  is  an isomorphism
$$\Pi_{G}^{0}(\{*\})\to  A(G)$$
\end{theorem}

\subsection{Computational  Remarks}
 We  introduce  some   basic  tools of  algebraic  topology  in order  to perform  computations. 

\begin{remark}[The  spectral sequence]
Let  $X$ be  a  proper  $G$-CW  complex.  There  is  an
equivariant Atiyah-Hirzebruch  spectral sequence  which 
converges  to $\Pi^{n}_{G} (X)$ and  whose $E^{2}$-term is  given
in terms  of Bredon cohomology

$$E_{2}^{p,q}= H^{p} _{\mathbb{Z}{\rm Or}G }(X,  \Pi_{?}^{q}) $$
AppLied  to the  universal  proper $G$-space $\EfinG$: 

$$E_{2}^{p,q}= H^{p} _{\mathbb{Z} \mathcal{SUB}_{\mathcal{COM}(G)}}
(\EfinG,  \Pi_{?}^{q}) $$

where $\Pi_{?}^{0} $  is  the  contravariant coefficient  system
  $H\mapsto \pi_{H}^{0}$. 

There  is  a  canonical identification 
$$H^{0} _{\mathbb{Z} \mathcal{SUB}_{\mathcal{COM}(G)}} (\EfinG,
  \Pi_{?}^{0})\cong {\varprojlim}_{H\in \mathcal{COM}} \Pi_{H}^{0}$$ 

The  edge homomorphism  of  the  spectral sequence defines  a map 
$${\rm edge}^{G}: \Pi_{G}^{0}(\EfinG)\to A_{inv}(G)$$

\end{remark}

\begin{definition}[An operator theoretical Burnside  Ring]
Let  $G$  be  a  locally  compact  group. The operator theoretical  burnside  ring of  $G$, $A^{\rm op}(G)$  is  the  $0$-dimensional  equivariant cohomotopy theory of the  classifying space  of  proper  actions $\EfinG$. In symbols
$$A^{\rm op} (G)=\Pi_{G}^{0}(\EfinG)$$

The  augmentation  ideal $\ideal_{G}\subset\Pi_{G}^{0}(\EfinG)$ is  defined  to  be the  kernel  of  the  composition of  the  restriction  to  the  oth- skeleton of  the  classifying  space  and  the restriction  to  the  trivial  group 

$$\Pi_{G}^{0}(\EfinG)\to \Pi_{G}^{0}(\EfinG_{0})\to \Pi_{\{e\}}^{0}(\EfinG_{0})$$
 
\end{definition}

 \begin{example}[The  group $Sl_{2}(\mathbb{R})$]
Recall that the  group $Sl_{2}(\mathbb{R})$  is  defined to  be  the
group  of real $2\times 2$-matrices  with  determinant 1. It  is  a
Lie  group of  dimension 3 and  has one  connected  component. The
maximal compact subgroup is  $S^{1}=SO_{2}$. 

As  $Sl_{2}(\mathbb{R})$  is  almost  connected,  a  model for  $\E
_{\mathcal{COM}}Sl_{2}$ is $Sl_{2}(\mathbb{R})/SO_{2}\approx
\mathbb{R}^{2}$, which can be handled as the  upper-half plane  model
for the $2$-dimensional hyperbolic  space. Note  that  this  is a 
zero-dimensional proper  $CW$-complex. From the  equivariant  Atiyah-
Hirzebruch  Spectral Sequence follows  that  the edge  homomorphism
$${\rm edge }^{Sl_{2}(\mathbb{R})}:
\Pi_{Sl_{2}(\mathbb{R})}^{0}(\EfinG Sl_{2}(\mathbb{R}))
\to \lim_{ {\rm inv}H \in \Komp  (Sl_{2}(\mathbb{R})) }\Pi_{H}(pt)$$
is  an isomorphism. On the  other  hand, since $S^{1}$ is  a final
object  in the  category  of  compact  subgroups of
$Sl_{2}(\mathbb{R})$, we  have
$$A^{\rm  op} (Sl_{2}(\mathbb{R}))\cong A(S^{1})$$ 

$$A (S^{1}) \cong  \mathbb{Z}$$
 is  a  well known fact. 
\end{example}
\subsection{Extending  one version  of the Segal conjecture for  lie  groups}
The  segal  conjecture for  finite groups, proven in  1984  \cite{carlssonsegal} states  the existence  of  an isomorphism  between  a  certan  completion  of  the  Burnside  ring   and  the  $0$th- stable  cohomotopy  of  the  classifying  space:

\begin{theorem}[Carlsson, 1984]
The  Segal  conjecture  is  true  for  finite  groups, that  is, there
is  an  isomorphism
$$A(G)_{\hat{I}_{G}} \cong \Pi^{0}_{1}({\rm B }G)$$
\end{theorem}
In the  case  of  Lie  groups, this  conjecture is  known  to  be  false  in this  statement  \cite{feshbachsegal}. However,  a weaker version  was obtained  by Feshbach  in \cite{feshbachsegal}  and  later  refined  by Bauer  in  \cite{bauersegal}. The  statement  is :

\begin{theorem}\label{theorem segalconjecturecompact} [Segal
    conjecture for compact  Lie  groups]

Let  $G$  be  a  compact  Lie group with  maximal  torus  $T$ of
dimension $n$ and  Weyl group $W=N_{T,G}/T$. Let $\rho: W \to
Gl_{n}(\mathbb{Z})$  be a  representation which gives rise  to the
action of  $W$  on $T \approx \mathbb{R}^{n}/\mathbb{Z}^{n}$. Suppose
that $\rho$ does  not  originate at  a generalized  quaternion group
of  order $2^{n}$. Then the  map 
$$A(G)_{\ideal_{G}}\to \Pi_{1}^{0}(\B G)  $$ 
has  dense image  in the  skeletal filtration.    
\end{theorem}
 We  extend  this  theorem, (mainly  to to  illustrate  the  interactions  of  our  methods  with  previous results)  in the  following   direction: 
 
 \begin{theorem}\label{theorem segalconjecturenoncompact} [Segal
    Conjecture for almost  connected  Lie groups] 
The  Segal conjecture  is  true  for (non compact) Lie  groups  with
finitely  many  components. That  is, there  is  a  map 
$$A^{\rm op} (G)_{\ideal_{G}} \to \Pi_{1}^{0}(\B G)$$
 with  dense  image  in the skeletal filtration  whenever a  maximal
 compact subgroup of  $G$  satisfies  the  hypotheses  of  theorem
 \ref{theorem   segalconjecturecompact}.   
 \end{theorem}
\begin{proof}
Let $G$  be  a Lie  group with  finitely  many  components. Then
\begin{enumerate}
\item{There  is up  to conjugacy  a unique  maximal compact subgroup $K$ of
$G$. Any  other  compact  subgroup  is subconjugated  to $K$.}
\item{There exist  diffeomorphisms $G \approx G/K\times K$
 and  $G/K\approx \mathbb{R}^{k}$.} 
\end{enumerate}
See \cite{hochschild}, theorem  3.1 p  180 for a  proof  of  this. Hence,  the  space  $G/K$  carries  a  proper  action, and  in  particular, the induction
isomorphism gives an  isomorphism $A^{op}(G)\cong A(K)$, where $A(K)$  stands  for
the  Burnside  ring  in the  sense  of  Tom  Dieck\cite{tomdieckburnside}. On the  other  hand, the  classifying  spaces $\B G$ and  $\B K$ have  the  same  homotopy  type. Hence  the  map  
$${\Pi_{G}^{0}}_{\ideal_{G}} (\EfinG) \overset{\cong}{\to} \Pi_{K}^{0}(\{*\})_{\ideal_{K}}\to \Pi_{e}^{0}(\B K)\overset{\cong}{\to} \Pi_{1}^{0}(\B G) $$ 
has  dense image  as a consequence  of theorem \ref{theorem segalconjecturecompact}.

\end{proof}

\typeout{-------------------------------------- References  ---------------------------------------}

\bibliography{nonlinearityjulio}
\bibliographystyle{abbrv}

\end{document}